\documentclass[submission]{FPSAC2019}
\usepackage[utf8]{inputenc}
\usepackage{tikz-cd}
\usepackage{graphicx}
\usepackage{subfig}

\DeclareMathOperator{\id}{id}

\DeclareMathOperator{\cano}{cano}

\newtheorem{theorem}{Theorem}
\newtheorem{lemma}[theorem]{Lemma}
\newtheorem{proposition}[theorem]{Proposition}
\newtheorem{corollary}[theorem]{Corollary}
\newtheorem{definition}[theorem]{Definition}
\newtheorem{example}[theorem]{Example}

\usepackage{lipsum}

\title[Polynomial invariants of hypergraphs]{Polynomial invariants and reciprocity theorems for the Hopf monoid of hypergraphs and its sub-monoids}

\author{Jean-Christophe Aval\addressmark{1}, Théo Karaboghossian\addressmark{1}, \and Adrian Tanasa\addressmark{1} \addressmark{2} \addressmark{3}}

\address{\addressmark{1}Univ. Bordeaux, Bordeaux INP, CNRS, LaBRI, UMR5800, F-33400 Talence, France, EU\\ \addressmark{2}IUF Paris, France, EU \\ \addressmark{3}H. Hulubei Nat. Inst. Phys. Nucl. Engineering Magurele, Romania, EU}

\received{\today}

\abstract{In arXiv:1709.07504 Aguiar and Ardila give a Hopf monoid structure on hypergraphs as well as a general construction of polynomial invariants on Hopf monoids. Using these results, we define in this paper a new polynomial invariant on hypergraphs. We give a combinatorial interpretation of this invariant on negative integers which leads to a reciprocity theorem on hypergraphs. Finally, we use this invariant to recover well-known invariants on other combinatorial objects (graphs, simplicial complexes, building sets etc) as well as the associated reciprocity theorems.}

\resume{Dans arXiv:1709.07504 Ardila et Aguiar ont donne une structure de monoïde de Hopf sur les hypergraphes ainsi qu'un moyen de construire des invariants polynomiaux sur les monoïdes de Hopf. En s'appuyant sur ces résultats, on définit dans ce papier un nouvel invariant polynomial d'hypergraphes. Nous montrons ensuite que cet invariant est sujet à un théorème de réciprocité en en donnant une interprétation combinatoire sur les entiers négatifs. Finalement on montre qu'il est possible de retrouver des invariants connus d'autre objets combinatoires (graphes, complexes simpliciaux etc) à l'aide de cet invariant.}


\begin{document}

\maketitle

\section{Introduction}

In combinatorics, Hopf structures give an algebraic framework to deal with operations of merging (product) and splitting (coproduct) combinatorial objects. The notion of Hopf algebra is well known and used in combinatorics for over 30 years, and has proved its great strength in various questions (see for example \cite{HAC}). More recently, Aguiar and Mahajan defined a notion of Hopf monoid \cite{am2},\cite{am} akin to the notion of Hopf algebra and built on Joyal’s theory of species \cite{joy}.
Such as in the case of Hopf algebras, a useful application of Hopf monoids is to define and compute polynomial invariants (see \cite{ABS}, \cite{BB}, \cite{DKT} or \cite{KMT} for various examples), as was put to light by the recent and extensive paper of Aguiar and Ardila \cite{AA}. In particular they give a theorem to generate various polynomial invariants and use it to recover the chromatic polynomial of graphs, the Billera-Jia-Reiner polynomial of matroids and the strict order polynomial of posets. Furthermore they also give a way to compute these polynomial invariants on negative integers hence also recovering the different reciprocity theorems associated to these combinatorial objects.

In this paper, we apply Aguiar and Ardila's theorem to the Hopf monoid of hypergraphs defined in \cite{AA}. This Hopf structure is different than the one defined and studied  in \cite{HypB} (the respective coproducts are different).
We obtain a combinatorial description for the (basic) invariant $\chi_I(H)(n)$ in terms of colorings of hypergraphs (Theorem \ref{chi}). We then use another approach (rather technical) than the method of \cite{AA} to get a reciprocity theorem for hypergraphs (Theorem \ref{chi-n}).
Finally, we use these results to obtain polynomial invariants on sub-monoids of the Hopf monoid of hypergraphs.

This paper is an extended abstract, all the detailed proofs can be found in the paper \cite{papier} as well as more results on sub-monoids of the Hopf monoid of hypergraphs.

\section{Definitions and reminders}
\subsection{Hopf monoids}

We present here basic definitions on Hopf monoids. The interested reader may refer to \cite{am} and to \cite{AA} for more information on this subject. In this paper $\Bbbk$ is a field and all vector spaces are over $\Bbbk$.

\begin{definition}
A \emph{vector species} $P$ consists of the following data.
\begin{itemize}
\item For each finite set $I$, a vector space $P[I]$.
\item For each bijection of finite sets $\sigma: I\rightarrow J$, a linear map $P[\sigma]:P[I]\rightarrow P[J]$. These maps should be such that $P[\sigma\circ\tau] = P[\sigma]\circ P[\tau]$ and $P[\id] = \id$.
\end{itemize}

A \emph{sub-species} of a vector species $P$ is a vector species $Q$ such that for each finite set $I$, $Q[I]$ is a sub-space of $P[I]$ and for each bijection of finite sets $\sigma: I\rightarrow J$, $Q[\sigma] = P[\sigma]_{|Q[I]}$.

A \emph{morphism} $f: P\rightarrow Q$ between vector species is a collection of linear maps $f_I : P[I] \rightarrow Q[I]$ satisfying the naturality axiom: for each bijection $\sigma: I\rightarrow J$, $f_J\circ P[\sigma] = Q[\sigma]\circ f_I$. 
\end{definition}

\begin{definition}
A \emph{connected Hopf monoid in vector species} is a vector species $M$ with $M[\emptyset] = \Bbbk$ that is equipped with product and co-product linear maps
\begin{displaymath}
\mu_{S,T}: M[S]\otimes M[T] \rightarrow M[S\sqcup T] \qquad \Delta_{S,T}: M[S\sqcup T] \rightarrow M[S]\otimes M[T],
\end{displaymath}
with $S$ and $T$ disjoint sets, and subject to the following axioms:
\begin{itemize}
\item \emph{Naturality}. For each pair of disjoint sets $S$, $T$, each bijection $\sigma$ with domain $S\sqcup T$, we have $M[\sigma]\circ\mu_{S,T} = \mu_{\sigma(S),\sigma(T)}\circ M[\sigma_{|S}]\otimes M[\sigma_{|T}]$ and $M[\sigma_{|S}]\otimes M[\sigma_{|T}]\circ\Delta_{S,T} = \Delta_{\sigma(S),\sigma(T)}\circ M[\sigma]$.
\item \emph{Unitality}. For each set $I$, $\mu_{I,\emptyset}$, $\mu_{\emptyset,I}$, $\Delta_{I,\emptyset}$ and $\Delta_{\emptyset,I}$ are given by the canonical isomorphisms $M[I]\otimes\Bbbk\cong\Bbbk\cong\Bbbk\otimes M[I]$.
\item \emph{Associativity}. For each triplet of pairwise disjoint sets $R$,$S$, $T$ we have: $\mu_{R,S\sqcup T}\circ\id\otimes\mu_{S,T}\linebreak[4] = \mu_{R\sqcup S, T}\circ \mu_{R,S}\otimes\id$.
\item \emph{Co-associativity}. For each triplet of pairwise disjoint sets $R$,$S$, $T$ we have: $\Delta_{R,S}\otimes\id\circ\Delta_{R\sqcup S, T} = \id\otimes\Delta_{S,T}\circ\Delta_{R,S\sqcup T}$.
\item \emph{Compatibility}. For each pair of disjoint sets $A$, $B$, each pair of disjoint sets $C$, $D$ we have the following commutative diagram, where $\tau$ maps $x\otimes y$ to $y\otimes x$:
\begin{center}
\begin{tikzcd}
P[S]\otimes P[T] \arrow[d, "\Delta_{A,B}\otimes\Delta_{C,D}" swap] \arrow[r, "\mu_{S,T}"] & P[I] \arrow[r, "\Delta_{S',T'}"] & P[S']\otimes P[T']  \\
P[A]\otimes P[B]\otimes P[C]\otimes P[D] \arrow[rr, "\id\otimes\tau\otimes\id" swap] & & P[A]\otimes P[C]\otimes P[B]\otimes P[D] \arrow[u, "\mu_{A,C}\otimes\mu_{B,D}" swap]
\end{tikzcd}
\end{center}
\end{itemize}

A \emph{sub-monoid} of a Hopf monoid $M$ is a sub-species of $M$ stable under the product and co-product maps. 


A \emph{morphism of Hopf monoids} in vector species is a morphism of vector species which preserves the products, co-products (compatibility axiom) and the unity (unitality axiom).
\end{definition}

We will use the term Hopf monoid for connected Hopf monoid in vector species. A sub-monoid of a Hopf monoid $M$ is itself a Hopf monoid when equipped with the product and co-product maps of $M$. We consider this to be always the case.

A \emph{decomposition} of a finite set $I$ is a sequence of pairwise disjoint sets $S=(S_1,\dots,S_l)$ such that $I=\sqcup_{i=1}^lS_i$. A \emph{composition} of a finite set $I$ is a decomposition of $I$ without empty parts. We will note $S \vdash I$ for $S$ a decomposition of $I$, $S\vDash I$ if $S$ is a composition, $l(S) = l$ the length of a decomposition and $|S| = |I|$ the number of elements in the decomposition.

\begin{definition}
Let $M$ a be a Hopf monoid. The \emph{antipode} of M is the morphism of Hopf monoids $\text{S} : M\rightarrow M$ defined by: $$ \text{S}_I = \sum_{\substack{(S_1,\dots S_k)\vDash I\\ k\geq 1}} \mu_{S_1,\dots,S_k}\circ\Delta_{S_1,\dots S_k}$$ for any finite set $I$.
\end{definition}

\begin{definition}
A \emph{character} on a Hopf monoid $M$ is a collection of linear maps $\zeta_I:M[I]\rightarrow \Bbbk$ subject to the following axioms.
\begin{itemize}
\item \emph{Naturality}. For each bijection $\sigma: I\rightarrow J$ we have $\zeta_J\circ M[\sigma] = \zeta_I$.
\item \emph{Multiplicativity}. For each disjoint sets $S$, $T$ we have $\zeta_{S\sqcup T}\circ\mu_{S,T} = m \circ\zeta_S\otimes\zeta_T$, where $m$ is the canonical isomorphism $\Bbbk\otimes\Bbbk\cong\Bbbk$.
\item \emph{Unitality}. $\zeta_{\emptyset}(1) = 1$.
\end{itemize}
\end{definition}

Let us recall from \cite{AA} the results which we will use in the sequel.

\begin{theorem}[Proposition 16.1 and Proposition 16.2 in \cite{AA}]
\label{inv}
Let $M$ be a Hopf monoid and $\zeta$ a character on $M$. For $x\in M[I]$ and $n$ an integer we define:
\begin{displaymath}
\chi_I(x)(n)= \sum_{(S_1, \dots S_n) \vdash I} \zeta_{S_1}\otimes\dots\otimes\zeta_{S_n}\circ\Delta_{S_1,\dots S_n}(x).
\end{displaymath}
Then $\chi_I$ is a polynomial invariant in $n$ verifying:
\begin{itemize}
\item $\chi_I(x)(1) = \zeta(x)$,
\item $\chi_{\emptyset} = 1$ and $\chi_{S\sqcup T}(\mu(x\otimes y)) = \chi_S(x)\chi_T(y)$,
\item $\chi_I(x)(-n) = \chi_I(\text{S}_I(x))(n)$.
\end{itemize}
\end{theorem}

Let $M$ be a Hopf monoid. For $I$ a set and $x\in M[I]$ we call $x$ \emph{discrete} if $I=\{i_1,\dots, i_{|I|}\}$ and $x=\mu_{\{i_1\},\dots,\{i_{|I|}\}}x_1\otimes\dots\otimes x_{|I|}$ for $x_j\in M[\{i_j\}]$. Then the maps that send discrete elements onto 1 and other elements onto 0 give us a character of Hopf monoid. Following the terminology introduced in Section 17 of \cite{AA}, we call \emph{basic invariant of $M$} the polynomial invariant obtained by applying Theorem \ref{inv} with this character. We note $\chi^M$ this polynomial or just $\chi$ when $M$ is clear from the context.










\section{Basic invariant of hypergraphs}

In all the following, $I$ always denotes a finite set.

Our goal is to express the basic invariant of the Hopf monoid of hypergraphs defined in Section 20 of \cite{AA}. More specifically we intend to obtain a combinatorial interpretation of $\chi_I(x)(n)$ and $\chi_I(x)(-n)$.

In this context, an \emph{hypergraph over $I$} is a collection of (possibly repeated) subsets of $I$, which we call \emph{edges}\footnote{in some references, the terms \emph{hyperedge} or \emph{multiedge} is used.}, containing $\emptyset$ exactly once. The elements of $I$ are then called \emph{vertices} of $H$ and $HG[I]$ denotes the free vector space of hypergraphs over $I$. Note that two hypergraphs over different sets can never be equal, e.g $\{\{1,2,3\},\{2,3,4\}\}\in HG[[4]]$ is not the same as $\{\{1,2,3\},\{2,3,4\}\}\in HG[[4]\cup\{a,b\}]$. This is illustrated in Figure \ref{same-edges}

\begin{figure}[htbp]
\begin{center}
\includegraphics[scale=1.5]{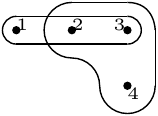}
\hspace{5cm}
\includegraphics[scale=1.5]{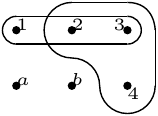}
\caption{Two hypergraphs with the same edges but over different sets.}
\label{same-edges}
\end{center}
\end{figure}

The product and co-product are given by, for $I=S\sqcup T$:
\begin{align*}
\mu_{S,T}: HG[S]\otimes HG[T] &\rightarrow HG[I] & \Delta_{S,T}: HG[I] &\rightarrow HG[S]\otimes HG[T] \\
H_1\otimes H_2 &\mapsto H_1\sqcup H_2 & H &\mapsto H_{|S}\otimes H_{/S}
\end{align*}
where $H_{|S} = \{e\in H\, |\, e\subseteq S\}$ is the \emph{restriction of $H$ to $S$} and $H_{/S} = \{e\cap T\, |\, e \nsubseteq S\}\cup\{\emptyset\}$ is the \emph{contraction of $S$ from $H$}. The discrete hypergraphs are then the hypergraphs with edges of cardinality at most 1.

\begin{example}
For $I=[5]$, $S=\{1,2,5\}$ and $T=\{3,4\}$, we have:
\begin{center}
\begin{tikzcd}[column sep=huge]
\raisebox{-21pt}{\includegraphics[scale=1.5]{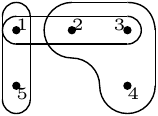}} \arrow[r, "\Delta_{S,T}"] & \raisebox{-10pt}{\includegraphics[scale=1.5]{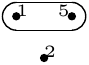}}\hspace{0.2cm} \otimes \hspace{0.2cm}\raisebox{-10pt}{\includegraphics[scale=1.5]{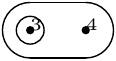}}
\end{tikzcd}
\end{center}
\end{example}


In \cite{AA}, Ardila and Aguiar propose a method to obtain a combinatorial interpretation of any polynomial invariant given by Theorem \ref{inv} on negative integers, assuming that we have an interpretation of it on positive integers. Their method consists in using a cancellation-free grouping-free formula for the antipode and the third point of Theorem \ref{inv}.
We use here a different approach: we express the polynomial dependency of $\chi_I(x)(n)$ in $n$ with a well chosen family of polynomials which we then use to calculate $\chi_I(x)(-n)$ and interpret the resulting formula.

Let us begin by giving a proposition giving us the aforementioned family of polynomials.
For $t\in\mathbb{N}^*$ and a sequence of positive integers $p_1, p_2,\dots, p_t$, we define $F_{p_1,\dots, p_t}$ as a function over the integers given by, for $n\in\mathbb{N}$:$$F_{p_1,\dots p_t}(n) = \sum_{0\leq k_1 <\dots < k_t\leq n-1} k_1^{p_1}\dots k_t^{p_t}.$$

\begin{proposition}
Let $p_1, p_2,\dots, p_t$ be integers. Then $F_{p_1,\dots p_t}$ is a polynomial of degree $\sum_{i=1}^t p_i + t$ whose coefficients can be expressed in term of sums and products of Bernoulli numbers (the explicit expression can be found in \cite{papier}).
\end{proposition}


Before stating our results on $\chi_I(H)(n)$ we need to introduce some definitions.
There exists a canonical bijection between decompositions and functions with co-domain of the form $[n]$. In the sequel, we will want to seamlessly pass from one notion to the other. We hence give a few explanations on this bijection. Given an integer $n$, the canonical bijection between decompositions of $I$ of size $n$ and functions from $I$ to $[n]$ is given by:
\begin{align*}
b_{I,n}: \{f:I\rightarrow [n]\} &\rightarrow \{P\vdash I\,|\, l(P)=n\} \\
f &\mapsto (f^{-1}(1),\dots,f^{-1}(n)).
\end{align*}
If it is clear from the context what are $I$ and $n$ we will write $b$ instead of $b_{I,n}$. If $P$ is a partition we will also refer to $b^{-1}(P)$ by $P$ so that instead of writing "i such that $v\in P_i$" we can just write $P(v)$. Similarly, if $P$ is a function we will refer to $b(P)$ by $P$ so that $P_i = P^{-1}(i)$. Also remark that $b_{I,n}$ induces a bijection between compositions of $I$ of size $n$ and surjections from $I$ to $[n]$.

\begin{definition}
Let $H$ be a hypergraph over $I$ and $n$ be an integer. A \emph{coloring of $H$ with $[n]$} is a function from $I$ to $[n]$ (or a decomposition of $I$ of length $n$ from what precede) and in this context the elements of $[n]$ are called \emph{colors}.

Let $S\vdash I$ be a coloring of $H$. For $v\in e\in H$, we say that $v$ is a \emph{maximal vertex of $e$ (for $S$)} if $v$ is of maximal color in $e$ and we call \emph{maximal color of $e$ (for $S$)} the color of a maximal vertex of $e$. We say that a vertex $v$ is a \emph{maximal vertex (for $S$)} if it is a maximal vertex of an edge.

If $J\subset I$ is a subset of vertices, the \emph{order of appearance of $J$ (for $S$)} is the composition $\cano(S_{|J})$ where $S_{|J} = (S_1\cap J,\dots, S_{l(S)}\cap J)$ and where te map $\cano$ sends any decomposition on the composition obtained by dropping the empty parts.
\end{definition}

\begin{example}
\label{hgcolo}
We represent the coloring of a hypergraph on $I=\{a,b,c,d,e,f\}$ with $\{${\color{Green}1},{\color{Blue}2},{\color{Magenta}3},{\color{Red}4}$\}$:
\begin{center}
\includegraphics[scale=1.6]{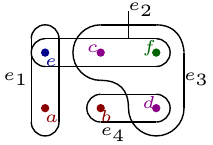}
\end{center}
The maximal vertex of $e_1$ is $a$ and the maximal vertices of $e_3$ are $c$ and $d$. The maximal color of $e_2$ is {\color{Magenta}{3}}. The order of appearance of $\{a,c,d,e\}$ is $(\{e\},\{c,d\},\{a\})$. The order of appearance of all edges is $(\{e_2,e_3\},\{e_1,e_4\})$.
\end{example}

\begin{definition}
Let $H$ be a hypergraph over $I$. An \emph{orientation} of $H$ is a function $f$ from $H$ to $I$ such that $f(e)\in e$ for every edge $e$. A \emph{directed cycle} in an orientation $f$ of $H$ is a sequence of distinct edges $e_1,\dots, e_k$ such that $f(e_1)\in e_2\setminus f(e_2),\dots, f(e_k)\in e_1\setminus f(e_1)$. An orientation is \emph{acyclic} if it does not have any cycle. We note $\mathcal{A}_H$ the set of acyclic orientation of $H$.

An orientation $f$ of $H$ and a coloring $S$ of $H$ with $[n]$ are said to be \emph{compatible} if $S(f(e)) = \max(S(e))$ for every $e\in H$. They are said to be \emph{strictly compatible} if $f(e)$ is the unique maximal vertex of $e$.
\end{definition}

We can now state our first theorem.

\begin{theorem}
\label{chi}
Let $I$ be a set and $H\in HG[I]$ a hypergraph over $I$. Then $\chi_I(H)(n)$ is the number of colorings of $H$ with $[n]$ such that every edge has only one maximal vertex. This is also the number of strictly compatible pairs of acyclic orientations and colorings with $[n]$.
Furthermore, defining $P_{H,f} = \{ P \vDash f(H)\, |\,  v\in e\setminus f(e) \Rightarrow P(v)< P(f(e))\}$, for every $f\in \mathcal{A}_H$, 
we have that $$\chi_I(H)(n) = n^{|J_H|}\sum_{f\in \mathcal{A}_H}\sum_{P\in P_{H,f}} F_{p_1,\dots, p_{l(P)}}(n),$$ where $J_H\subset I$ is the set of isolated vertices of $H$ (i.e vertices not in an edge) and for every $P\in P_{H,f}$, $p_i =|\tilde{P}_i|$ and $\tilde{P}_i = \left(\bigcup_{e\in f^{-1}(P_i)}e\right) \cap f(H)^c\bigcap_{j<i} \tilde{P}_j^c$.
\end{theorem}

\begin{proof}[Sketch of proof]
For $S$ a decomposition of $I$ of size $n$, note $H_1\otimes\dots\otimes H_n = \Delta_{S_1,\dots, S_n}(H)$. Let $S$ be a decomposition of $I$ of size $n$. Let $e$ be an edge. We then have the equivalence:
\begin{align*}
e\in H_i &\iff e\cap S_i \not = \emptyset \wedge\forall j>i, e\cap S_j = \emptyset \\
&\iff \text{$e\cap S_i$ is the set of maximal vertices of $e$}
\end{align*}
Hence, we have that 
\begin{align*}
\zeta_{S_1}\otimes\dots\otimes\zeta_{S_n}\circ\Delta_{S_1,\dots, S_n}(H) = 1 &\iff \forall e\in H, e\in H_i \Rightarrow |e\cap S_i| = 1\\
&\iff\text{each edge has only one maximal vertex.}
\end{align*}
The equivalence between the colorings such that every edge has only one maximal vertex and the strictly compatible pairs of acyclic orientations and colorings is given by the bijection $S\mapsto (e\mapsto v_e,S)$, where $v_e$ is the unique vertex in $e$ such that $S(v_e) = \max(S(e))$.

The term $n^{|J_H|}$ in the formula is trivially obtained, in the following we hence consider that $H$ has no isolated vertices.

Informally, for a hypergraph with no isolated vertex, the formula can be obtained by the following reasoning. To choose a coloring such that every edge has only one maximal vertex, one can proceed in the following:
\begin{enumerate}
\item choose the maximal vertex of each edge ($f\in \mathcal{A}_H$),
\item choose in which order those vertices appear ($P\in P_{H,f}$),
\item choose the color of those vertices ($k_1+1,\dots,k_{l(P)}+1$), (and notice that the set of such choices is empty if $l(P)> n$, which allow us to not add this non polynomial dependency in $n$ at the previous choice),
\item choose the colors of the yet uncolored vertices which are in the same edge than a vertex of minimal color in $f(H)$ ($k_1^{|\tilde{P}_1|}$); then those in the same edge than a vertex of second minimal color in $f(H)$ ($k_2^{|\tilde{P}_2|}$), etc.
\end{enumerate}

\end{proof}

\begin{example}
The coloring given in Example \ref{hgcolo} is not counted in $\chi_I(H)(4)$ since $e_3$ has two maximal vertices. However by changing the color of $d$ to {\color{Blue}2} we do obtain a coloring where every edge has only one maximal vertex.
\end{example}

We are now interested in the value of $(-1)^{|I|}\chi_I(H)(-n)$. We first state two lemmas.


\begin{lemma}
\label{f-n}
Let $p=(p_1,\dots, p_t)$ be a sequence of integers. Then $$F_p(-n) = (-1)^{d_t}\sum_{p\prec q}F_q(n+1)$$ where $p\prec q$ stands for $q$ coarsens $p$ as a composition of $\sum_i p_i$. 
\end{lemma}

\begin{definition}
Let $P = (P_1,\dots,P_l)\vDash I$ and $Q = (Q_1,\dots, Q_k) \vDash J$ be two compositions of two disjoint sets $I$ and $J$. The \emph{product} of $P$ and $Q$ is the composition $P\cdot Q = (P_1,\dots,P_l,Q_1,\dots Q_k)$. The \emph{shuffle product} of $P$ and $Q$ is the set $sh(P,Q) = \{R\vDash I\sqcup J\,|\, P=\cano (R_{|I}), Q=\cano(R_{|J})\}$.

Let $P'\vDash I$ be another composition of $I$. We say that $P'$ refines $P$ and note $P'\prec P$ if $P' = Q_1 \cdot\dots\cdot Q_l$ with $Q_i$ a composition of $P_i$.
\end{definition}

This next lemma is the crux of the following Theorem 3. Its (combinatorial, and -we hope- nice) proof may be found in [5].

\begin{lemma}
\label{compsum}
Let $I$ be a set and $P\vDash I$ a composition of $I$. We have the identity: $$\sum_{Q\prec P} (-1)^{l(Q)} = (-1)^{|P|}.$$

Let furthermore $G$ be a directed acyclic graph on $I$ and consider the \emph{constrained set} \linebreak[4] $C(G,P)=\{Q\prec P\,|\, \forall (v,v')\in G,Q(v)< Q(v')\}$. We have the more general identity: $$\sum_{Q\in C(G,P)} (-1)^{l(Q)} = \left\{\begin{array}{cl}
0  & \text{if there exists $(v,v')\in G$ such that $P(v')<P(v)$},    \\ 
(-1)^{|P|}  & \text{else.}\end{array}\right.$$
\end{lemma}

We can now state the main result of this section:

\begin{theorem}[Reciprocity theorem on hypergraphs]
\label{chi-n}
Let $I$ be a set and $H\in HG[I]$ a hypergraph over $I$. Then $(-1)^{|I|}\chi_I(H)(-n)$ is the number of compatible pairs of acyclic orientations and colorings with $[n]$ of $H$. In particular, $(-1)^{|I|}\chi_I(H)(-1) = |\mathcal{A}_H|$ is the number of acyclic orientations of $H$.
\end{theorem}

\begin{proof}[Sketch of proof]
From Proposition \ref{chi} and Lemma \ref{f-n} we have that $$\chi_I(H)(-n) = (-n)^{|J_H|}\sum_{f\in \mathcal{A}_H}\sum_{P\in P_{H,f}}(-1)^{\sum_{i=1}^{l(P)}p_i + l(P)}\sum_{(p_1, \dots, p_{l(P)}) \prec q}F_q(n+1).$$
Remark that:
\begin{itemize}
\item $\sum_{i=1}^{l(P)}p_i = |I\setminus J_H| -|f(H)|$ (since $(\tilde{P}_1,\dots \tilde{P}_{l(P)}, f(H))$ is a partition of $I\setminus J_H$),
\item $\phi :\{Q\vDash f(H)\,|\, P\prec Q\} \rightarrow \{\text{$q$ composition of $|I\setminus J_H| -|f(H)|$}\,|\, (p_1,\dots, p_{l(P)}) \prec q\}$ $Q\mapsto (|\tilde{Q}_1|,\dots, |\tilde{Q}_{l(Q)}|)$ is a bijection.
\end{itemize}
This gives us:
\begin{align*}
(-1)^{|I|}\chi_I(H)(-n) &= n^{|J_H|}\sum_{f\in \mathcal{A}_H}(-1)^{|f(H)|}\sum_{P\in P_{H,f}}(-1)^{l(P)}\sum_{P\prec Q} F_{\phi(Q)}(n+1)\\
&= n^{|J_H|}\sum_{f\in \mathcal{A}_H}(-1)^{|f(H)|}\sum_{Q\vDash f(H)}\left(\sum_{\substack{P\prec Q\\P\in P_{H,f}}} (-1)^{l(P)}\right)F_{\phi(Q)}(n+1).
\end{align*}
By definition of $\mathcal{A}_H$, $G = \{(v,f(e))\,|\, v\in e\setminus f(e)\}$ is a directed acyclic graph on $f(H)$. Hence, remarking that $\{ P\prec Q\,|\, P\in P_{H,f}\} = C(G,Q)$, Lemma \ref{compsum} leads to:
\begin{align*}
(-1)^{|I|}\chi_I(H)(-n) &= n^{|J_H|}\sum_{f\in \mathcal{A}_H}(-1)^{|f(H)|}\sum_{\substack{P\vDash f(H)\\ P(v)\leq P(v') \forall(v,v')\in G}} (-1)^{|f(H)|}F_{\phi(P)}(n+1) \\
&= n^{|J_H|}\sum_{f\in \mathcal{A}_H}\sum_{P\in P_{H,f}'} F_{\phi(P)}(n+1) \\
&= n^{|J_H|}\sum_{f\in \mathcal{A}_H}\sum_{P\in P_{H,f}'}F_{p_1,\dots, p_{l(P)}}(n+1),
\end{align*}
where $P_{H,f}' = \{ P\vDash f(H)\,|\, P(v\in e\setminus f(e) \leq P(f(e)\}$.

The end of the proof is analogous to the proof of Theorem \ref{chi}.
\end{proof}

\begin{example}
For any $I$ and any $H\in HG[I]$, we have $\chi_I(H)(n) \leq (-1)^{|I|}\chi_I(H)(-n)$. This comes from the fact that any strictly compatible pair is compatible.

The coloring given in Example \ref{hgcolo} has two compatible acyclic orientations: both send $e_1$ on $a$, $e_2$ on $c$ and $e_4$ on $b$ but one sends $e_3$ on $c$ and the other $e_3$ on $d$.

For the color set $\{${\color{Blue}1},{\color{Red}2}$\}$, the following coloring has 4 compatible orientations but only two are acyclic.
\begin{center}
\includegraphics[scale=1.6]{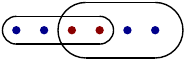}
\end{center}
\end{example}

\section{Application to other Hopf monoids}

In this section we use Theorem \ref{chi} and Theorem \ref{chi-n} to obtain a combinatorial interpretation of the basic invariants for some Hopf monoids presented in Sections 20 to 25 of \cite{AA}. The results presented in the subsections 4.2 to 4.4 already appear, at least implicitly, in previous works (details are provided at the beginning of each these subsections).



In all the following, we denote by $\chi$ the basic invariant of the Hopf monoid of hypergraphs.

\subsection{Simple hypergraphs}
\label{sshg}
A hypergraph is \emph{simple} if it has no repeated edges. The vector species $SHG$ of simple hypergraphs is not stable by the contraction defined on hypergraphs but it still admits a Hopf monoid structure. The product and co-product are given by, for $I=S\sqcup T$:
\begin{align*}
\mu_{S,T}: SHG[S]\otimes SHG[T] &\rightarrow SHG[I] & \Delta_{S,T}: SHG[I] &\rightarrow SHG[S]\otimes SHG[T] \\
H_1\otimes H_2 &\mapsto H_1\sqcup H_2 & H &\mapsto H_{|S}\otimes H_{/S},
\end{align*}
where $H_{|S} = \{e\in H\, |\, e\subseteq S\}$ and $H_{/S} = \{e\cap T\, |\, e \nsubseteq S\}\cup\{\emptyset\}$ but this time without repetition, i.e $H_{/S}$ can also be defined as $\{B\subseteq\,|\, \exists A\subseteq S, A\sqcup B \in H\}$. A discrete simple hypergraph is then a simple hypergraph with edges of cardinality at most one.

\begin{proposition}
\label{shg}
$\chi^{SHG}$ is the restriction of $\chi$ to the vector species of simple hypergraphs.
\end{proposition}


\subsection{Graphs}
\label{graphs}
The result of this subsection has already been given in Section 18 of \cite{AA}, but we give it here as a consequence of our result in the previous section. 

A \emph{graph} can be seen as a hypergraph whose edges are all of cardinality 2. The vector species $G$ of graphs is not stable by the contraction defined on hypergraphs, but it still admits a Hopf monoid structure. The product and co-product are given by, for $I=S\sqcup T$:
\begin{align*}
\mu_{S,T}: G[S]\otimes G[T] &\rightarrow G[I] & \Delta_{S,T}: G[I] &\rightarrow G[S]\otimes G[T] \\
g_1\otimes g_2 &\mapsto g_1\sqcup g_2 & g &\mapsto g_{|S}\otimes g_{/S},
\end{align*}
where $g_{|S}$ is the sub-graph of $g$ induced by $S$ and $g_{/S} = g_{|T}$. A discrete graph is then a graph with no edges.

A \emph{proper coloring} of a graph is a coloring such that no edge has its two vertices of the same color. The chromatic polynomial of a graph is the polynomial $T$ such that $T(n)$ is the number of proper colorings with $n$ colors.

\begin{corollary}
The basic invariant of $G$ is the chromatic polynomial.
\end{corollary}


In particular, by evaluating $\chi$ on negative integers for a graph, we recover the classical reciprocity theorem of Stanley \cite{stan}.

\subsection{Simplicial complexes}
In \cite{sc} Benedetti, Hallam, and Machacek constructed a combinatorial Hopf algebra of simplicial complexes and in particular they obtained results which generalise those given in this subsection. 

An \emph{abstract simplicial complex}, or simplicial complex, on $I$ is a collection $C$ of subsets of $I$, called \emph{faces}, such that any subset of a face is a face i.e $J \in C$ and $K \subset J$ implies $J \in C$. The \emph{1-skeleton} of a simplicial complex is the graph formed by its faces of cardinality 2.

By Proposition 21.1 of \cite{AA}, the vector species $SC$ of simplicial complexes is a sub-monoid of the Hopf monoid of simple hypergraphs.

\begin{corollary}
Let $I$ be a set, $C\in SC[I]$ and $g$ its 1-skeleton. Then $\chi^{SC}_I(C)$ is the chromatic polynomial of $g$.
\end{corollary}

\subsection{Paths}
\label{paths}
Proposition 25.7 of \cite{AA} states that the Hopf monoid of set of paths is isomorphic to the Hopf monoid of associahedra which is a sub-monoid of a quotient of the Hopf monoid of generalized permutahedra. Hence, it should be possible to deduce the result of this subsection from \cite{AA}.

A \emph{word} on $I$ is a total ordering of $I$. The \emph{paths} on $I$ are the words on $I$ quotiented by the relation $w_1\dots w_{|I|} \sim w_{|I|}\dots w_1$. A \emph{set of paths} $\alpha$ of $I$ is a partition $(I_1,\dots, I_l)$ of $I$ with a path $s_i$ on each part $I_i$ and we will note $\alpha = s_1|\dots|s_l$. The vector species $F$ of sets of paths admits a Hopf monoid structure, the product and co-product are given by, for $I=S\sqcup T$:
\begin{align*}
\mu_{S,T}: F[S]\otimes F[T] &\rightarrow F[I] & \Delta_{S,T}: F[I] &\rightarrow F[S]\otimes F[T] \\
\alpha_1\otimes \alpha_2 &\mapsto \alpha_1\sqcup \alpha_2 & \alpha &\mapsto \alpha_{|S}\otimes \alpha_{/S}
\end{align*}
where if $\alpha = s_1|\dots|s_l$, $\alpha_{|S}= s_1\cap S|\dots| s_l\cap S$ forgetting the empty parts and $\alpha_{/S}$ is the set of paths obtained by replacing each occurrence of an element of $S$ in $\alpha$ by the separation symbol $|$. A discrete set of paths is then a set of paths where all paths have only one element.

\begin{example}
For $I=\{a,b,c,d,e,f,g\}$ and $S=\{b,c,e\}$ and $T=\{a,d,f,g\}$ we have: $$\Delta_{S,T}(bfcg|aed) = bc|e\otimes f|g|a|d$$
\end{example}


\begin{example}
For $I=\{a,b,c,d,e,f,g\}$ and $\alpha=bfcg|aed$, $l(\alpha)$ is the following graph:
\begin{center}
\includegraphics[scale=1.5]{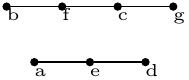}
\end{center}
\end{example}

We consider that a binary tree is (\emph{strictly}) \emph{compatible} with a coloring if each vertex is of color (strictly) greater than its children.

\begin{corollary}
Let $I$ be a set and $\alpha$ be a path on $I$. Then $\chi^F_I(\alpha)(n)$ is the number of strictly compatible pairs of binary trees with $|I|$ vertices and colorings with $[n]$ and $\chi^F_I(\alpha)(-n)$ is the number of compatible pairs of binary trees with $|I|$ vertices and colorings with $[n]$. In particular $\chi^F_I(\alpha)(-1) = C_{|I|}$ where $C_n = \frac{1}{n+1}\binom{2n}{n}$ is the $n$-th Catalan number.
\end{corollary}

\bibliographystyle{plain}
\bibliography{abstract}

\begin{thebibliography}{10}

\bibitem{AA}
Marcelo Aguiar and Federico Ardila.
\newblock Hopf monoids and generalized permutahedra, 2017.

\bibitem{ABS}
Marcelo Aguiar, Nantel Bergeron, and Frank Sottile.
\newblock Combinatorial {H}opf algebras and generalized {D}ehn-{S}ommerville
  relations.
\newblock {\em Compositio Mathematica}, 142:1--30, 2006.

\bibitem{am2}
Marcelo Aguiar and Swapneel Mahajan.
\newblock {\em Monoidal functors, species and {H}opf algebras}, chapter~15.
\newblock Providence, R.I. : American Mathematical Society, 2010.

\bibitem{am}
Marcelo Aguiar and Swapneel Mahajan.
\newblock Hopf monoids in the category of species.
\newblock {\em American Mathematical Society}, 585, 2013.

\bibitem{papier}
Jean-Christophe Aval, Théo Karaboghossian, and Adrian Tanasa.
\newblock Basic invariants and reciprocity theorems for the {H}opf monoid of
  hypergraphs and its sub-monoids, 2018.

\bibitem{HypB}
Carolina Benedetti, Nantel Bergeron, and John Machacek.
\newblock Hypergraphic polytopes: combinatorial properties and antipode, 2017.

\bibitem{sc}
Carolina Benedetti, Joshua Hallam, and John Machacek.
\newblock Combinatorial {H}opf algebras of simplicial complexes.
\newblock {\em S{I}{A}{M} {J}ournal on {D}iscrete Mathematics},
  30(3):1737–1757, 2016.

\bibitem{BB}
Nantel Bergeron and Carolina Benedetti.
\newblock Cancelation free formula for the antipode of linearized {H}opf
  monoid, 2016.

\bibitem{DKT}
G.~H. Duchamp, N.~{Hoang-Nghia}, T.~Krajewski, and A.~Tanasa.
\newblock Renormalization group-like proof of the universality of the {T}utte
  polynomial for matroids, 2013.

\bibitem{HAC}
Darij Grinberg and Victor Reiner.
\newblock Hopf {A}lgebras in {C}ombinatorics, 2016.

\bibitem{joy}
André Joyal.
\newblock Une théorie combinatoire des séries formelles.
\newblock {\em Advances in Mathematics}, 42(1):1--82, 1981.

\bibitem{KMT}
Thomas Krajewski, Iain Moffatt, and Adrian Tanasa.
\newblock Hopf algebras and {T}utte polynomials.
\newblock {\em Advances in {A}pplied {M}athematics}, 95:271--330, 2018.

\bibitem{stan}
Richard~P. Stanley.
\newblock Acyclic orientations of graphs.
\newblock {\em Discrete {M}ath}, 5:171–178, 1973.

\end{thebibliography}
\end{document}